\documentclass[12pt]{amsart}

\usepackage{amsfonts}
\usepackage{amsmath}

\newtheorem{thm}{Theorem}[section]

\newtheorem{cor}[thm]{Corollary}

\newtheorem{prop}[thm]{Proposition}
\newtheorem*{prob*}{Problem}
\newtheorem*{thm*}{Theorem}

\theoremstyle{definition}
\newtheorem{defn}[thm]{Definition}

\newtheorem*{defn*}{Definition}
\newtheorem{rem}[thm]{Remark}
\numberwithin{equation}{section}

\newcommand{\C}{\mathbb C}
\newcommand{\R}{\mathbb R}
\newcommand{\Y}{\mathbb Y}
\newcommand{\Z}{\mathbb Z}
\newcommand{\Zp}{\mathbb Z_{\geq 0}}

\newcommand{\X}{\mathfrak{X}}

\DeclareMathOperator{\Conf}{Conf} 
\DeclareMathOperator{\Plancherel}{Pl}

\DeclareMathOperator{\Prob}{Prob} 
 
\DeclareMathOperator{\Pf}{Pf} \DeclareMathOperator{\PF}{Pf}
\DeclareMathOperator{\TX}{\widetilde{X}}

\begin{document}
\title[Pfaffian $L$-ensembles related to measures on partitions]
 {\bf{Pfaffian $L$-ensembles related to the $z$-measures
on partitions with the Jack parameters $\theta=\frac{1}{2}, 2$.}}

\author{Eugene Strahov}

\thanks{Department of Mathematics, The Hebrew University of
Jerusalem, Givat Ram, Jerusalem 91904. E-mail:
strahov@math.huji.ac.il. Supported by US-Israel Binational Science
Foundation (BSF) Grant No. 2006333,
 and by Israel Science Foundation (ISF) Grant No. 0397937.\\
}


\commby{}
\begin{abstract}
We construct   Pfaffian $L$-ensembles related to the $z$-measures on
partitions and to the Plancherel measures on partitions with the
Jack parameters $\theta=\frac{1}{2}, 2$.
 The results imply that these measures on partitions lead to
 Pfaffian point processes, and the correlation kernels of these
 processes can be expressed in terms of the  corresponding $L$-matrices. We give
 explit formulae for these $L$-matrices.
\end{abstract}
\maketitle
\section{Introduction}
It is known that measures on partitions arising in the context of
the representation theory of the infinite symmetric group lead to
determinatal point processes. The most well known example is the
Plancherel measure on partitions studied in many papers, see, for
example, Logan and Shepp \cite{logan}, Vershik and Kerov
\cite{vershik11}, Baik, Deift and Johansson
\cite{BaikDeiftJohansson-increasing
subsequences,BaikDeiftJohansson-Second row}, Deift
\cite{Deift-Combinatorial Theory}, Ivanov and Olshanski
\cite{ivanov}. The relation with determinantal point processes was
established in Borodin, Okounkov, and Olshanski in Ref.
\cite{BorOkounkovOlshanski}, and by Johansson in Ref.
\cite{johansson-discrete orthogonal ensembles}. To explain this
relation let us identify partitions of $n$ with Young diagrams
containing $n$ boxes. The set of such Young diagrams will be
denoted as $\Y_n$. Let $\Y=\Y_0\cup\Y_1\cup\ldots $ be the set of
all Young diagrams including the empty diagram $\emptyset$.
Consider the poissonized Plancherel measure $M_{\Plancherel,\eta}$
on $\Y$ obtained by mixing together the Plancherel measures
$M^{(n)}_{\Plancherel}$ on $\Y_n$, $n=0,1,2,\ldots $,
$$
M_{\Plancherel,\eta}(\lambda)=e^{-\eta}\frac{\eta^{|\lambda|}}{|\lambda|!}M^{(|\lambda|)}_{\Plancherel}(\lambda),\;\lambda\in\Y.
$$
Here $\eta>0$. Under the correspondence
$\lambda\rightarrow\{\lambda_i-i+\frac{1}{2}\}$ the poissonized
Plancherel measure $M_{\Plancherel,\eta}$  turns into the
determinantal point process on the lattice $\Z+\frac{1}{2}$ whose
correlation kernel is the discrete Bessel kernel, see Theorems 1, 2
in  Borodin, Okounkov, and Olshanski \cite{BorOkounkovOlshanski}.

The $z$-measures $M_{z,z'}^{(n)}$ on partitions  is another
important class of measures leading to determinantal point
processes on $\Z+\frac{1}{2}$.  These measures are parameterized
by two complex parameters $z$, $z'$, and were first introduced in
Kerov, Olshanski, and Vershik \cite{KOV1} in the context of the
harmonic analysis on the infinite symmetric group. The
$z$-measures $M_{z,z'}^{(n)}$ were studied in detail by Borodin
and Olshanski \cite{BO-hyper}-\cite{BO-ZmeasuresScalingLimits},
see also related papers by Okounkov \cite{okounkov5}, Borodin,
Olshanski, and Strahov \cite{BorodinOlshanskiStrahov-Giambelli}.
The relation with determinantal point processes arises in a way
similar to that in the case of the Plancherel measure. Namely,
consider the mixed $z$-measures $M_{z,z',\xi}$ with an additional
parameter $\xi\in (0,1)$ obtained by mixing up the $z$-measures
$M_{z,z'}^{(n)}$,
$$
M_{z,z',\xi}(\lambda)=(1-\xi)^{zz'}\frac{(zz')_{|\lambda}|\xi^{|\lambda|}}{|\lambda|!}
M_{z,z'}^{|\lambda|}(\lambda),
$$
where $\lambda$ ranges over $\Y$. Then under the correspondence
$\lambda\rightarrow\{\lambda_i-i+\frac{1}{2}\}$ the measure
$M_{z,z',\xi}$ turns into a determinantal point process on the
lattice $\Z+\frac{1}{2}$. The correlation kernel of this point
process can be written in terms of the Gaussian hypergeometric
functions, see Borodin and Olshanski \cite{BO-hyper}.

In the examples described above  the  determinantal point processes
have a special remarkable feature: they can be understood as
 determinantal $L$-ensembles.  To define a determinantal $L$-ensemble
 let us introduce a finite set (called a phase space)
  $\X$,
 and let $L$ be a $|\X|\times|\X|$ matrix (called the $L$-matrix) whose rows and columns
 are parameterized by points of $\X$. If $L$ is positive definite,
 one can define a random point process on $\X$ by
 $$
 \Prob\{X\}=\frac{\det L(X|X)}{\det(I+L)},
 $$
 where $X$ is a subset of $\X$, and $L(X|X)$ is the symmetric
 submatrix of $L$ corresponding to $X$. It is a well-known fact
 that such $L$-ensemble is a determinantal point process whose
 correlation kernel $K$ is given by $K=L(I+L)^{-1}$.
 The definition of determinantal $L$-ensembles can be extended to infinite phase
 spaces $\X$ provided that the determinant $\det(I+L)$ is well
 defined. For other properties of determinantal $L$-ensembles and their applications
 we
 refer the reader to Borodin \cite{Borodin-DeterminantalPointProcesses}, Section 5.

For the poissonized Plancherel measure $M_{\Plancherel,\eta}$, and
for the mixed $z$-measures $M_{z,z',\xi}$ the corresponding
$L$-matrices can be written explicitly in terms of elementary
functions only, and these matrices have remarkably simple forms.
Moreover, it turns out that the kernels $L(x,y)$ defining the
$L$-matrices are integrable in the sense of Its, Izergin, Korepin,
and Slavnov \cite{its}. This leads to an algorithm (based on
Riemann-Hilbert problems) to compute explicitly the correlation
kernel $K$, see Borodin \cite{Borodin-RiemannHilbert}.

Kerov \cite{kerov}, Borodin and Olshanski
\cite{BO-ZmeasuresScalingLimits} have shown that it is natural to
consider a deformation $M_{z,z',\theta}^{(n)}$ of
$M_{z,z'}^{(n)}$, where $\theta>0$ is called the parameter of
deformation (or the Jack parameter). Such deformations are in many
ways similar to log-gas (random-matrix) models with arbitrary
$\beta=2\theta$. As in the theory of log-gas models the value
$\beta=2$ is a distinguished one and leads to determinantal point
processes. On the next level of difficulty are the cases
$\theta=2$ or $\theta=1/2$ ($\beta=4$ or $\beta=1$, respectively).
In these cases the measures $M_{z,z',\theta}^{(n)}$ lead to
Pfaffian point processes, similar to ensembles of Random Matrix
Theory of $\beta=4$ or $\beta=1$ symmetry types, see Borodin and
Strahov \cite{Borodin-Strahov-Ensembles}, Strahov
\cite{Strahov-Matrix-Kernel}-\cite{Strahov-z-measures} for the
available results in this direction. It turns out that such
Pfaffian point processes are of great interest to the harmonic
analysis on the infinite symmetric group. The fact that these
measures play a role in the harmonic analysis was established by
Olshanski \cite{olshanskiletter}, and the detailed explanation of
this representation-theoretic aspect can be found in Strahov
\cite{Strahov-Z-measures-Gelfandpairs}.

The aim of this work is to show that the Pfaffian point processes
related to the $z$-measures on partitions with the Jack parameters
$\theta=\frac{1}{2}, 2$, and to the Plancherel measures on
partitions with the Jack parameters $\theta=\frac{1}{2}, 2$ can be
understood as Pfaffian $L$-ensembles (see Borodin and Rains
\cite{BorodinRains-EynardMehtaTheorem} , Borodin and Strahov
\cite{BorodinStrahov-CharacteristicPolynomials} and Section
\ref{SectionPfaffianLEnsembles} of the present paper for a
definition of  Pfaffian $L$-ensembles). The paper gives explicitly
the $L$-matrices for these ensembles, see Theorem
\ref{MainTheoremTheta2} and Theorem \ref{MainTheoremThetaHalf}.

In the context of the harmonic analysis on the infinite symmetric
group the most important problem is to understand the scaling
limits of the determinantal and of the Pfaffian point processes
defined by the mixing $z$-measures $M_{z,z',\xi,\theta}$
($\theta=1,\frac{1}{2}$, or $2$), as $\xi\nearrow 1$. A possible
approach to this problem is to study the convergence of the
corresponding (determinantal or Pfaffian) $L$-ensemble to that
defined by a limiting bounded operator $\mathcal{L}$ acting in
$L^2(\R\setminus\{\emptyset\})\oplus
L^2(\R\setminus\{\emptyset\})$. The limiting operator
$\mathcal{L}$ completely characterizes the limiting point process
relevant for the harmonic analysis. In particular, the correlation
kernel of the limiting process can be expressed in terms of
$\mathcal{L}$. For $\theta=1$ such a limit transition was
investigated by Borodin \cite{Borodin-RiemannHilbert} (see also
Borodin and Olshanski \cite{BorOlsh-Point processes} and the
references therein).

The results of the present paper (explicit formulae for the
$L$-matrices defining the Pfaffian $L$-ensembles for
$M_{z,z',\xi,\theta=\frac{1}{2}}$ and $M_{z,z',\xi,\theta=2}$) lay a
foundation for this approach in the Pfaffian case. The author plans
to investigate the transition to such limiting ensembles in a
subsequent publication.

The paper is organized as follows.  Section
\ref{SectionPfaffianLEnsembles} contains the definition and some
basic properties of Pfaffian $L$-ensembles.  Section
\ref{Sectionztheta} contains the definitions of the $z$-measures
and the Plancherel measures on partitions with an arbitrary Jack
parameter $\theta>0$. The main results of the present work are
stated in Section \ref{SectionMainResults}, see Theorem
\ref{MainTheoremTheta2} and Theorem \ref{MainTheoremThetaHalf}.
Section \ref{SectionSpecialClass} investigates the properties of a
special class of Pfaffian $L$-ensembles, which is relevant for
measures on partitions considered in this paper. In Sections
\ref{SectionPlancherelLebsenbles} and
\ref{SectionZmeausuresLebsenbles} we rewrite the $z$-measures and
Plancherel measures on partitions with  the Jack parameters
$\theta=\frac{1}{2}, 2$ in terms of suitable (Frobenius-type)
coordinates. Then  we use formulae obtained in Section
\ref{SectionSpecialClass} to prove Theorem \ref{MainTheoremTheta2}
and Theorem \ref{MainTheoremThetaHalf}.

\section{Pfaffian $L$-ensembles}\label{SectionPfaffianLEnsembles}
Let $\X$ be a countable set. Given  $\X$ let us construct two copies
of $\X$, and denote them by $\X'$ and $\X''$. Each point $x\in\X$
has a prototype $x'\in\X'$ and another one $x''\in\X''$. Let $L$ be
a $|\X|\times |\X|$ skew symmetric matrix constructed from $2\times
2$ blocks with rows and columns parameterized by elements of
$\X'\times\X''$. The $2\times 2$ blocks have the form
\begin{equation}\label{PLl}
L(x,y)=\left[\begin{array}{cc}
  L(x',y') & L(x',y'')\\
  L(x'',y') & L(x'',y'') \\
\end{array}\right].
\end{equation}
Once $x,y$ take values in $\X$, the variables $x'$, $y'$ ($x''$,
$y''$) are the elements of $\X'$ ($\X''$) corresponding to $x$,
$y$. The matrix $L$ can also be  understood as a $2|\X|\times
2|\X|$ matrix with rows and columns parameterized by points of
$\X$.

Let $\Conf(\X)$ be the set of all subsets of $\X$ and denote by
$\Conf(\X)_0\subset\Conf(\X)$ the set of finite subsets of $\X$. To
any $X\subset\Conf(\X)_0$ there will correspond a $2\times 2$ block
antisymmetric submatrix of $L$ of a finite size. We denote this
submatrix by $L(X\vert X)$. If $X$ consists of $m$ points,
\begin{equation}
X=\left(x_1,\ldots , x_m\right),\;\; X\in \Conf(\X)_0\nonumber,
\end{equation}
then the submatrix $L(X\vert X)$ has the form
\begin{equation}\label{PLl1}
L(X\vert X)=
\left[%
\begin{array}{ccccc}
  0 & L(x_1',x_1'') & \ldots & L(x_1',x_m') & L(x_1',x_m'') \\
  -L(x_1',x_1'') & 0 &  &L(x_1'',x_m') & L(x_1'',x_m'') \\
  \vdots & & &  &  \\
  -L(x_1',x_m') & -L(x_1'',x_m')  &  & 0 & L(x_m',x_m'') \\
  -L(x_1',x_m'') & -L(x_1'',x_m'') &  & -L(x_m',x_m'') & 0 \\
\end{array}%
\right]. \nonumber
\end{equation}
Denote by $\mbox{Pf}\;A$ the Pfaffian of an even dimensional
antisymmetric matrix $A$.  In what follows we always assume that the
matrix $L$ has the
 property
\begin{equation}
\Pf\;L(X\vert X)\geq 0,\;\;\;\forall X\in\Conf(\X)_0 ,
\end{equation}
and
\begin{equation}\label{Condition}
\sum\limits_{X: X\in\Conf(\X)_0}\Pf\;L(X\vert X)<\infty.
\end{equation}
Let $J$ be $2\times 2$ block matrix of format $\X\times \X$ with
matrix elements
\begin{equation}
J(x,y)=\left\{%
\begin{array}{ll}
    \left[\begin{array}{cc}
      0 & 1 \\
      -1 & 0 \\
    \end{array}\right], & x=y; \\
    0, & \hbox{otherwise}. \\
\end{array}%
\right.
\end{equation}
Define the expression $\Pf\;(J+L)$ by the formula
\begin{equation}\label{Pf(J+L)}
\Pf\;(J+L)=\sum\limits_{X: X\in\Conf(\X)_0}\Pf\;L(X\vert X).
\end{equation}
Condition (\ref{Condition}) ensures that the sum in the righthand
side of equation (\ref{Pf(J+L)}) is finite.  Note that if $\X$ is a
finite set, then $L$ and $J$  are matrices  of finite size, and
equation (\ref{Pf(J+L)}) is the expansion of $\Pf\;(J+L)$ into a sum
of Pfaffians of antisymmetric $2\times 2$ block submatrices $L(X|X)$
of $L$.
\begin{defn}
A point process on $\X$ defined by
\begin{equation}
\mbox{Prob}_{L}(X)=\dfrac{\Pf\;L(X\vert
X)}{\Pf\;(J+L)},\;\;\;\forall X\in\Conf(\X)_0,
\end{equation}
is called the Pfaffian $L$-ensemble.
\end{defn}
The fact that
$\sum\limits_{X\in\subset\Conf(\X)_0}\mbox{Prob}_{L}(X)=1$ follows
from equation (\ref{Pf(J+L)}).

By correlation functions  $\varrho(X)$ for the Pfaffian
$L$-ensembles we mean the probabilities that random configurations
include fixed sets $X$, namely
$$
\varrho(X)=\sum\limits_{Y:\;Y\in\Conf(\X)_0,\;Y\supseteq
X}\mbox{Prob}_{L}(Y).
$$
The striking property of the Pfaffian $L$-ensembles is that the
correlation function $\varrho(X)$ is given by a Pfaffian,
\begin{equation}
\varrho(X)=\Pf\left[K(x_i,x_j)\right]_{i,j=1}^m,\;\;X=(x_1,\ldots
,x_m)\in\Conf(\X)_0.
\end{equation}
Here the matrix $K$ is  a $|\X|\times |\X|$ skew symmetric matrix
made from $2\times 2$ blocks with rows and columns parameterized by
elements of $\X'\times\X''$, and defined in terms of $L$ by the
expression
\begin{equation}
K=J+(J+L)^{-1}.
\end{equation}
The Pfaffian expression for $m$-point correlation functions reflects
the fact that the Pfaffian $L$-ensembles is a special class of
Pfaffian point processes.

\section{The $z$-measures on partitions with the general
parameter $\theta>0$}\label{Sectionztheta}
 We use Macdonald
\cite{macdonald} as a basic reference for the notations related to
integer partitions and to symmetric functions. In particular, every
decomposition
$$
\lambda=(\lambda_1,\lambda_2,\ldots,\lambda_l):\;
n=\lambda_1+\lambda_2+\ldots+\lambda_{l},
$$
where $\lambda_1\geq\lambda_2\geq\ldots\geq\lambda_l$ are positive
integers, is called an integer partition. We identify integer
partitions with the corresponding Young diagrams, and denote the set
of all Young diagrams by $\Y$.  The set of Young diagrams with $n$
boxes  is denoted by $\Y_n$. Thus
$$
\Y=\bigcup\limits_{n=0}^{\infty}\Y_n.
$$

Following Borodin and Olshanski \cite{BO-ZmeasuresScalingLimits},
Section 1, let $M_{z,z',\theta}^{(n)}$ be a complex measure on
$\Y_n$ defined by
\begin{equation}\label{EquationVer4zmeasuren}
M_{z,z',\theta}^{(n)}=\frac{n!(z)_{\lambda,\theta}(z')_{\lambda,\theta}}{(t)_nH(\lambda,\theta)H'(\lambda,\theta)},
\end{equation}
where $n=1,2,\ldots $, and where we use the following notation
\begin{itemize}
    \item $z,z'\in\C$ and $\theta>0$ are parameters, the parameter
    $t$ is defined by
    $$
    t=\frac{zz'}{\theta}.
    $$
    \item $(t)_n$ stands for the Pochhammer symbol,
    $$
    (t)_n=t(t+1)\ldots (t+n-1)=\frac{\Gamma(t+n)}{\Gamma(t)}.
    $$
    \item
    $(z)_{\lambda,\theta}$ is a multidemensional analogue of the
    Pochhammer symbol defined by
    $$
    (z)_{\lambda,\theta}=\prod\limits_{(i,j)\in\lambda}(z+(j-1)-(i-1)\theta)
    =\prod\limits_{i=1}^{l(\lambda)}(z-(i-1)\theta)_{\lambda_i}.
    $$
     Here $(i,j)\in\lambda$ stands for the box in the $i$th row
     and the $j$th column of the Young diagram $\lambda$, and we
     denote by $l(\lambda)$ the number of nonempty rows in the
     Young diagram $\lambda$.
    \item
    $$
    H(\lambda,\theta)=\prod\limits_{(i,j)\in\lambda}\left((\lambda_i-j)+(\lambda_j'-i)\theta+1\right),
   $$
   $$
     H'(\lambda,\theta)=\prod\limits_{(i,j)\in\lambda}\left((\lambda_i-j)+(\lambda_j'-i)\theta+\theta\right),
   $$
      where $\lambda'$ denotes the transposed diagram.
\end{itemize}
\begin{prop}\label{PropositionHH}
The following symmetry relations hold true
$$
H(\lambda,\theta)=\theta^{|\lambda|}H'(\lambda',\frac{1}{\theta}),\;\;(z)_{\lambda,\theta}
=(-\theta)^{|\lambda|}\left(-\frac{z}{\theta}\right)_{\lambda',\frac{1}{\theta}}.
$$
Here $|\lambda|$ stands for the number of boxes in the diagram
$\lambda$.
\end{prop}
\begin{proof}
These relations follow immediately from definitions of
$H(\lambda,\theta)$ and $(z)_{\lambda,\theta}$.
\end{proof}
\begin{prop}\label{PropositionMSymmetries}
We have
$$
M_{z,z',\theta}^{(n)}(\lambda)=M_{-z/\theta,-z'/\theta,1/\theta}^{(n)}(\lambda').
$$
\end{prop}
\begin{proof}
Use definition of $M_{z,z',\theta}^{(n)}(\lambda)$, equation
(\ref{EquationVer4zmeasuren}), and apply Proposition
\ref{PropositionHH}.
\end{proof}
\begin{prop}\label{Prop1.3}
We have
$$
\sum\limits_{\lambda\in\Y_n}M_{z,z',\theta}^{(n)}(\lambda)=1.
$$
\end{prop}
\begin{proof}
See Kerov \cite{kerov}, Borodin and Olshanski
\cite{BO-ZmeasuresScalingLimits,BOHARMONICFUNCTIONS}.
\end{proof}
\begin{prop}\label{PropositionPositivity}
If parameters $z, z'$ satisfy one of the three conditions listed
below, then the measure $M_{z,z',\theta}^{(n)}$ defined by
expression (\ref{EquationVer4zmeasuren}) is a probability measure on
$Y_n$. The conditions are as follows.\begin{itemize}
    \item Principal series: either
$z\in\C\setminus(\Z_{\leq 0}+\Zp\theta)$ and $z'=\bar z$.
    \item The complementary series: the parameter $\theta$ is a rational number, and both $z,z'$
are real numbers lying in one of the intervals between two
consecutive numbers from the lattice $\Z+\Z\theta$.
    \item The degenerate series: $z,z'$ satisfy one of the
    following conditions\\
    (1) $(z=m\theta, z'>(m-1)\theta)$ or $(z'=m\theta,
    z>(m-1)\theta)$;\\
    (2) $(z=-m, z'<-m+1)$ or $(z'=-m,
    z<m-1)$.
\end{itemize}
\end{prop}
\begin{proof} See Propositions 1.2, 1.3 in Borodin and Olshanski
\cite{BO-ZmeasuresScalingLimits}.
\end{proof}
In what follows we fix two complex parameters $z$, $z'$ such that
the conditions in Proposition \ref{PropositionPositivity} are
satisfied, and $M_{z,z',\theta}^{(n)}$ is a probability measure on
$\Y_n$.

It is convenient  to mix all measures $M_{z,z',\theta}^{(n)}$, and
to define a new measure $M_{z,z',\xi,\theta}$  on
$\Y=\Y_0\cup\Y_1\cup\ldots $. Namely, let $\xi\in(0,1)$ be an
additional parameter, and set
\begin{equation}\label{EquationMzztheta}
M_{z,z',\xi,\theta}(\lambda)=(1-\xi)^t\xi^{|\lambda|}
\frac{(z)_{\lambda,\theta}(z')_{\lambda,\theta}}{H(\lambda,\theta)H'(\lambda,\theta)}.
\end{equation}
\begin{prop} We have
$$
\sum\limits_{\lambda\in\Y}M_{z,z',\xi,\theta}(\lambda)=1.
$$
\end{prop}
\begin{proof}
Follows immediately from Proposition \ref{Prop1.3}.
\end{proof}
Thus $M_{z,z',\xi,\theta}(\lambda)$ is a probability measure on
$\Y$. We will refer to $M_{z,z',\xi,\theta}(\lambda)$ as  the
mixed $z$-measure with the deformation (Jack) parameter $\theta$.

When both $z,z'$ go to infinity, expression
(\ref{EquationVer4zmeasuren}) has a limit
\begin{equation}\label{EquationPlancherelInfy}
M_{\Plancherel,\theta}^{(n)}(\lambda)=\frac{n!\theta^{n}}{H(\lambda,\theta)H'(\lambda,\theta)}
\end{equation}
called the Plancherel measure on $\Y_n$ with general $\theta>0$.
Proposition \ref{PropositionMSymmetries} implies that
\begin{equation}\label{PlancherelSymmetries}
M_{\Plancherel,\theta}^{(n)}(\lambda)=M_{\Plancherel,\frac{1}{\theta}}^{(n)}(\lambda').
\end{equation}
Instead of (\ref{EquationPlancherelInfy}), sometimes it is more
convenient to consider the Poissonized Plancherel measure with
general $\theta>0$,
\begin{equation}\label{EquationPlancherelInfyMixed}
M_{\Plancherel,\eta,\theta}(\lambda)=e^{-\eta}\left(\eta\right)^{|\lambda|}
\frac{\theta^{|\lambda|}}{H(\lambda,\theta)H'(\lambda,\theta)},
\end{equation}
where $\eta>0$.  Clearly, $M_{\Plancherel,\eta,\theta}(\lambda)$ is
a probability measure on the set $\Y$.
\begin{rem}
(1) Statistics of the Plancherel measure with the general Jack
parameter $\theta>0$ is discussed in   Matsumoto \cite{matsumoto}.
Matsumoto \cite{matsumoto} compares limiting distributions of rows
of random partitions with distributions of certain random
variables from a traceless Gaussian $\beta$-ensemble.\\
(2) When $\theta=1$ the poissonized Plancherel measure
$M_{\Plancherel,\eta,\theta}(\lambda)$, and the mixed $z$-measure
$M_{z,z',\xi,\theta}$ lead to dereminantal processes on
$\Z+\frac{1}{2}$, see Borodin, Okounkov and Olshanski
\cite{BorOkounkovOlshanski}, Johansson \cite{johansson-discrete
orthogonal ensembles}, Borodin and Olshanski \cite{BO-hyper}.\\
(3) The poissonized Plancherel measures and certain analogues of the
mixed $z$-measures on the strict partitions are considered in Petrov
\cite{petrov}. Petrov \cite{petrov} shows that such
measures lead to determinantal processes as well.\\
(4) For $\theta=\frac{1}{2}$ or $\theta=2$ the poissonized
Plancherel measure $M_{\Plancherel,\eta,\theta}(\lambda)$, and the
mixed $z$-measure $M_{z,z',\xi,\theta}$ lead to Pfaffian point
processes on $\Z+\frac{1}{2}$, see Strahov
\cite{Strahov-Matrix-Kernel, Strahov-z-measures}, and the
references therein.
\end{rem}
\section{Main results}\label{SectionMainResults}
\subsection{$L$-matrices}
Set $\X=\Z+\frac{1}{2}$, $\X_+=\Z_{\geq 0}+\frac{1}{2}$, and
$\X_-=\Z_{\leq 0}-\frac{1}{2}$. Introduce the parity on the sets
$\X_+=\Z_{\geq 0}+\frac{1}{2}$ and $\X_-=\Z_{\leq 0}-\frac{1}{2}$
referring to $\frac{1}{2}$ and $-\frac{1}{2}$  as to even
elements.

According to the decomposition of the set $\Z+\frac{1}{2}$ ,
$$
\Z+\frac{1}{2}=\Z_{\leq 0}-\frac{1}{2}\;\sqcup
\frac{1}{2}\sqcup\;\Z_{\geq 0}+\frac{3}{2},
$$
we write the matrix $L$ in the block form:
\begin{equation}\label{PstructureL}
L=\left[\begin{array}{ccc}
  L_{--} & L_{-0} & L_{-+} \\
  L_{0-} & L_{00} & L_{0+} \\
  L_{+-}& L_{+0} & L_{++} \\
\end{array}\right].
\end{equation}
We are interested in the matrices $L$ defined by
\begin{equation}\label{PmatrixL}
L=\left[\begin{array}{ccc}
  \mbox{E} & \mbox{A} & \mbox{B} \\
  -\mbox{A}^T & 0 & 0 \\
  -\mbox{B}^T & 0 & 0 \\
\end{array}\right]
\end{equation}
As usual, here E, A, B are the matrices with $2\times 2$ block
elements. Specifically,
\begin{equation}\label{PmatrixE}
\mbox{E}(x,y)=\left[\begin{array}{cc}
  \epsilon(x,y) & 0 \\
  0 & 0 \\
\end{array}\right],\;\;x,y\in\;\Z_{\leq 0}-\frac{1}{2},
\end{equation}
\begin{equation}\label{PmatrixA}
\mbox{A}(x,y)=\left[\begin{array}{cc}
  \epsilon(x,y) & 0 \\
  0 & \dfrac{h(x)h(y)}{x-y} \\
\end{array}\right],\;\;x\in\;\Z_{\leq 0}-\frac{1}{2},\;\; y=\frac{1}{2},
\end{equation}
\begin{equation}\label{PmatrixB}
\mbox{B}(x,y)=\left[\begin{array}{cc}
  0 & 0 \\
 \dfrac{h(x)h(y)}{x-y}  & \dfrac{h(x)h(y-\frac{1}{2})}{x- y+\frac{1}{2}} \\
\end{array}\right],\;\;x\in\;\Z_{\leq 0}-\frac{1}{2}, \;\;y\in\;\Z_{\geq 0}+\frac{3}{2}.
\end{equation}
The two-point function $\epsilon(x,y)$ in equations above is
antisymmetric, $\epsilon(x,y)=-\epsilon(y,x)$. When $x<y$,
\begin{equation}\label{Pepsilon}
\epsilon(x,y)=\left\{%
\begin{array}{ll}
    1, & x-\hbox{odd}, y-\hbox{even},\\
    0, & \hbox{otherwise.} \\
\end{array}%
\right.
\end{equation}
The function $h$ is nonnegative on $\Z+\frac{1}{2}$.
\subsection{Measures on partitions with the Jack parameter
$\theta=2$ as Pfaffian $L$-ensembles}\label{SectionMainResult2} We
define an embedding $\lambda\rightarrow X$ of the set $\Y$ of
Young diagrams into the set $\Conf(\Z+\frac{1}{2})_0$ of finite
configurations in $\Z+\frac{1}{2}$ as follows. Let $\lambda$ be a
Young diagram. Given $\lambda=(\lambda_1,\ldots,\lambda_l)$ we
denote by $\lambda\sqcup\lambda$  another Young diagram defined by
$$
\lambda\sqcup\lambda=\left(\lambda_1,\lambda_1,\lambda_2,\lambda_2,\ldots,\lambda_l,\lambda_l\right).
$$

Denote by $\left(P_1,\ldots,P_D\vert Q_1,\ldots,Q_D\right)$ the
usual Frobenius coordinates of $\lambda\sqcup\lambda$ (Macdonald
\cite{macdonald}, \S I.1). Thus $P_i$ is the number of squares in
the $i$th row to the right of the diagonal of
$\lambda\sqcup\lambda$, $Q_i$ is the number of squares in the $i$th
column below the diagonal of $\lambda\sqcup\lambda$, and
$i=1,\ldots, D$. Here $D$ is the number of boxes on the diagonal of
$\lambda\sqcup\lambda$. Given $\lambda$ consider the point
configuration $X=X_-\sqcup X_+$ on $\Z+\frac{1}{2}$ defined in terms
of the Frobenius coordinates $(P_1,\ldots,P_D\vert Q_1,\ldots, Q_D)$
as follows. If $D$ is even, then we set
\begin{equation}\label{X1}
X_+=\left(P_{D-1}+\frac{1}{2},
P_{D-3}+\frac{1}{2},\ldots,P_{1}+\frac{1}{2}\right).
\end{equation}
If $D$ is odd, then we set
\begin{equation}\label{X2}
X_-=\left(\frac{1}{2},P_{D-2}+\frac{1}{2},
P_{D-4}+\frac{1}{2},\ldots,P_{1}+\frac{1}{2}\right).
\end{equation}
(Observe that if $D$ is odd, then $P_D=0$.) In both cases (when
$D$ is either even or odd) we define $X_-$ as
\begin{equation}\label{X3}
X_-=\left(-Q_1-\frac{1}{2},
-Q_2-\frac{1}{2},\ldots,-Q_D-\frac{1}{2}\right).
\end{equation}
Equations (\ref{X1})-(\ref{X3}) define the embedding
$\lambda\rightarrow X$ of $\Y$ into $\Conf_0(\Z+\frac{1}{2})$.
Under this embedding any probability measure $M$ on $\Y$ turns
into a probability measure on $\Conf(\Z+\frac{1}{2})_0$. (Assume
that $X\in\Conf_0(\Z+\frac{1}{2})$, and assume that there is no a
Young diagram $\lambda$ such that $X$ is representable in terms of
Frobenius coordinates of $\lambda\sqcup\lambda$ by equations
(\ref{X1})-(\ref{X3}). Then we agree that the probability of $X$
is zero). Therefore we get a point process on $\Z+\frac{1}{2}$. We
will denote by $\underline{M}$ the point process obtained in this
way from a probability measure $M$ on $\Y$.

Let us introduce the following notation. For any complex $a$ and a
nonnegative integer $n$ we set
$$
[a]_n=\left\{%
\begin{array}{ll}
    (a+1)(a+3)\ldots (a+n-1), & n\;\hbox{is even,} \\
    a(a+2)\ldots (a+n-1), & n\;\hbox{is odd;} \\
    1, & n=0. \\
\end{array}%
\right.
$$
\begin{thm} \label{MainTheoremTheta2}(A) The point process $\underline{M}_{z,z',\xi,\theta=2}$ is
the Pfaffian $L$-ensemble in the sense of Section
\ref{SectionPfaffianLEnsembles}. The corresponding $L$-matrix is
defined by equations (\ref{PstructureL})-(\ref{Pepsilon}) with
\begin{equation}\label{hzmeasureTheta=2}
h(x)=\left\{%
\begin{array}{ll}
    \frac{[z+1]_{x-\frac{1}{2}}[z'+1]_{x-\frac{1}{2}}}{\Gamma(x+\frac{1}{2})}\xi^{\frac{x}{2}}, & x\in\Z_{\geq 0}+\frac{1}{2}, \\
     \frac{[-z]_{-x-\frac{1}{2}}[-z']_{-x-\frac{1}{2}}}{\Gamma(-x+\frac{1}{2})}\xi^{-\frac{x}{2}}, & x\in\Z_{\leq 0}-\frac{1}{2}.\\
\end{array}%
\right.
\end{equation}
Moreover, we have
$$
\Pf(J+L)=(1-\xi)^{-\frac{zz'}{2}}.
$$
(B) The point process $\underline{M}_{\Plancherel,\eta,\theta=2}$
is the Pfaffian $L$-ensemble in the sense of Section
\ref{SectionPfaffianLEnsembles}. The corresponding $L$-matrix is
defined by equations (\ref{PstructureL})-(\ref{Pepsilon}) with
\begin{equation}\label{hPlancherel}
h(x)=\frac{(2\eta)^{\frac{1}{2}(|x|+\frac{1}{2})}}{\Gamma(|x|+\frac{1}{2})},\;
x\in\Z+\frac{1}{2}.
\end{equation}
Moreover, we have
$$
\Pf(J+L)=e^{\eta}.
$$
\end{thm}
\subsection{Measures on partitions with the Jack parameter
$\theta=\frac{1}{2}$ as Pfaffian $L$-ensembles}

In this case we define an embedding $\lambda\rightarrow X'$ of the
set $\Y$ of Young diagrams into the set $\Conf(\Z+\frac{1}{2})_0$
of finite configurations in $\Z+\frac{1}{2}$ in a slightly
different way. Let $\lambda$ be a Young diagram. Denote by
$\left(P_1',\ldots,P_D'\vert Q_1',\ldots,Q_D'\right)$ the usual
Frobenius coordinates of $\lambda'\sqcup\lambda'$.  Given
$\lambda$ consider the point configuration $X'=X'_-\sqcup X'_+$ on
$\Z+\frac{1}{2}$ defined in terms of $(P_1',\ldots,P_D'\vert
Q_1',\ldots, Q_D')$ as follows. If $D$ is even, then we set
\begin{equation}\label{X1'}
X_+'=\left(P'_{D-1}+\frac{1}{2},
P'_{D-3}+\frac{1}{2},\ldots,P'_{1}+\frac{1}{2}\right).
\end{equation}
If $D$ is odd, then we set
\begin{equation}\label{X2'}
X'_-=\left(\frac{1}{2},P'_{D-2}+\frac{1}{2},
P'_{D-4}+\frac{1}{2},\ldots,P'_{1}+\frac{1}{2}\right).
\end{equation}
(Observe that if $D$ is odd, then $P'_D=0$.) In both case (when
$D$ is either even or odd) we define $X'_-$ as
\begin{equation}\label{X3'}
X'_-=\left(-Q'_1-\frac{1}{2},
-Q'_2-\frac{1}{2},\ldots,-Q'_D-\frac{1}{2}\right).
\end{equation}
Equations (\ref{X1'})-(\ref{X3'}) define the embedding
$\lambda\rightarrow X'$ of $\Y$ into $\Conf(\Z+\frac{1}{2})_0$.
Under the embedding  we get a point process on $\Z+\frac{1}{2}$.
We will denote by the same symbol $\underline{M}$ (as in Section
\ref{SectionMainResult2}) the point process obtained in this way
from a probability measure $M$ on $\Y$.
\begin{thm}\label{MainTheoremThetaHalf} (A) The point process $\underline{M}_{z,z',\xi,\theta=\frac{1}{2}}$ is
the Pfaffian $L$-ensemble in the sense of Section
\ref{SectionPfaffianLEnsembles}. The corresponding $L$-matrix is
defined by equations (\ref{PstructureL})-(\ref{Pepsilon}) with
\begin{equation}\label{hzmeasureTheta=Half}
h(x)=\left\{%
\begin{array}{ll}
    \frac{[-2z+1]_{x-\frac{1}{2}}[-2z'+1]_{x-\frac{1}{2}}}{\Gamma(x+\frac{1}{2})}\xi^{\frac{x}{2}}, & x\in\Z_{\geq 0}+\frac{1}{2}, \\
     \frac{[2z]_{-x-\frac{1}{2}}[2z']_{-x-\frac{1}{2}}}{\Gamma(-x+\frac{1}{2})}\xi^{-\frac{x}{2}}, & x\in\Z_{\leq 0}-\frac{1}{2}.\\
\end{array}%
\right.
\end{equation}
Moreover, we have
$$
\Pf(J+L)=(1-\xi)^{-2zz'}.
$$
(B) The point process
$\underline{M}_{\Plancherel,\eta,\theta=\frac{1}{2}}$ is the
Pfaffian $L$-ensemble in the sense of Section
\ref{SectionPfaffianLEnsembles}. The corresponding $L$-matrix is
the same as in Theorem \ref{MainTheoremTheta2}: it is defined by
equations (\ref{PstructureL})-(\ref{Pepsilon}) with the function
$h$ defined by equation (\ref{hPlancherel}).
\end{thm}

\section{Special class of Pfaffian
$L$-ensembles}\label{SectionSpecialClass} In this Section we
consider the Pfaffian $L$-ensemble on $\Z+\frac{1}{2}$ whose
$L$-matrix is defined by equations
(\ref{PstructureL})-(\ref{Pepsilon}).

Configurations $X\in\Conf_0(\Z+\frac{1}{2})$ can be divided into
two classes. The first class consists of configurations which do
not include the point $\frac{1}{2}$. Such configurations have the
form $X=X_-\sqcup X_+$, $X_+=(x_1^+,x_2^+,\ldots )$,
$x_1^+>\frac{1}{2}$. The second class consists of configurations
that include the point $\frac{1}{2}$. For such configuration
$X_+=(\frac{1}{2},x_1^+,x_2^+,\ldots )$. For any
$X\subset\Conf_0(\Z+\frac{1}{2})$ denote by $\TX$ the
configuration defined by
\begin{equation}
\TX=\TX_-\sqcup\TX_+, \nonumber
\end{equation}
\begin{equation}
\TX_-=X_-,\nonumber
\end{equation}
\begin{equation}
\TX_+=\left\{%
\begin{array}{ll}
    (x_1^+-1,x_1^+,x_2^+-1,x_2^+,\ldots ), &\;\;\; X_+\cap\frac{1}{2}=\emptyset, \\
     (\frac{1}{2},x_1^+-1,x_1^+,x_2^+-1,x_2^+,\ldots ), &\;\;\; X_+\cap\frac{1}{2}\neq\emptyset.\\
\end{array}%
\right. \nonumber
\end{equation}
\begin{defn}\label{DefinitionConfigurations}
We say that $X\in \Conf^{L}(\Z+\frac{1}{2})$ if
\begin{itemize}
\item $X\in \Conf_0(\Z+\frac{1}{2})$ \item
$X_+=(x_1^+<x_2^+<\ldots )$
    \item all points of $\TX_+$ are different
    \item $\vert \TX_-\vert=\vert \TX_+\vert$
    \item $\TX_-=(x_1^-<x_2^-<\ldots )$, where $x_i^-$ has the same
    parity as $i$.
\end{itemize}
\end{defn}
Let us introduce the following notation. Set
$\prod(A;B)\equiv\prod\limits_{i=1}^k\prod\limits_{j=1}^l(a_i-b_j)$
for any two sets $A=(a_1,\ldots ,a_k)$, $B=(b_1,\ldots ,b_l)$, let
$V(X)$ be the Vandermonde determinant associated with a set $X$,
\begin{equation}
V(X)=\prod\limits_{1\leq i<j\leq N}(x_i-x_j),\;\;\; X=(x_1,\ldots
,x_N),\nonumber
\end{equation}
 and  set $h(X)=\prod\limits_{j=1}^Nh(x_j)$.  Definition (\ref{DefinitionConfigurations})
 is
justified by the following statement.
\begin{prop}\label{PTheoremReductionToPfaffianProcesses} With $L$ given by equations
(\ref{PstructureL})-(\ref{Pepsilon}) we have
\begin{equation}\label{PReductionToPfaffianProcessesEquation}
\Pf L(X|X)=\dfrac{V(\TX_-)V(\TX_+)}{\prod(\TX_+;\TX_-)}\;h(\TX)
\end{equation}
for $X\in Conf^{L}(\Z+\frac{1}{2})$ and $0$ for all other $X\in
\Conf_0(\Z+\frac{1}{2})$.
\end{prop}
\begin{proof}
This fact was first proved in Borodin and Strahov
\cite{BorodinStrahov-CharacteristicPolynomials}, Section 3.2. We
reproduce here this proof (with minor changes) to make the paper
self-contained.

The positive integer $d=\vert\TX_-\vert=\vert\TX_+\vert$ can be
even or odd, depending on whether $X$ includes the point
$\frak x=\frac{1}{2}$ or not. So we consider two cases.\\
\textbf{Case 1}. $X\cap \frak x=\emptyset$\\
Given copies $X'$, $X''$ of $X\in\X$ in $\X'$, $\X''$ we denote by
$X'\uplus X''$ the set $(x_1', x_1'', x_2', x_2'',\ldots )$. Then
we have
\begin{equation}
\begin{split}
\PF\;L(X\vert X)&=\PF\;L\biggl[X_-\sqcup X_+\arrowvert X_-\sqcup
X_+\biggr]\\
&=\PF\;L\biggl[\left(X_-'\uplus X_{-}''\right)\sqcup
\left(X_{+}'\uplus X_{+}''\right)\arrowvert \left(X_-'\uplus
X_{-}''\right)\sqcup \left(X_{+}'\uplus X_{+}''\right)\biggr]\\
&=(-)^{\frac{d(d-1)}{2}}\cdot \PF\;L\biggl[X_-'\arrowvert X_-'\biggr]\\
&\times \PF\;L\biggl[ X_{-}''\sqcup \left(X_{+}'\uplus
X_{+}''\right)\arrowvert  X_{-}''\sqcup \left(X_{+}'\uplus
X_{+}''\right)\biggr]
\end{split}
\nonumber
\end{equation}
as the function $L(x,y)=0$ for any $x\in X_-'$ and any $y$ which
does not belong to $X_-'$ (see equations (\ref{PLl})-(\ref{PLl1})
and (\ref{PmatrixL})-(\ref{Pepsilon})). We note that $L(x,y)=0$,
if $x,y\in\;\X_-''$, or if $x,y\in X_+'\uplus X_+''$ Therefore
$\vert X_-''\vert =\vert X_+'\arrowvert +\vert X_+''\vert $, or
$\vert X_-\vert=2\vert X_+\vert $, which means that
$\vert\TX_+\vert =\vert \TX_-\vert $.

Consider $\PF\;L\biggl[X_-'\vert X_-'\biggr]$. Note that the
matrix $L\biggl[X_-'\vert X_-'\biggr]$ is even dimensional, if
$\vert X_- \vert=2\vert X_+\vert$. Moreover the matrix
$L\biggl[X_-'\vert X_-'\biggr]$ is the matrix whose $(i,j)$ entry
is, by definition, given by $\epsilon (x_i^-,x_j^-)$. Clearly, if
$x_1^-$ is even, the first row of this matrix consists of  zeros
only. Thus, if $\PF\;L\biggl[X_-'\vert X_-'\biggr]\neq 0$, $x_1^-$
must be odd. Now assume that $x_{2i-1}^-$ and $x_{2i}^-$ have the
same parity. In this case $(2i-1)^{\mbox{st}}$ and
$2i^{\mbox{th}}$ rows of the matrix $L\biggl[X_-'\vert
X_-'\biggr]$ are equal to each other. Therefore, if
$\PF\;L\biggl[X_-'\vert X_-'\biggr]\neq 0$ the elements of the set
$\TX_-=(x_1^-,x_2^-,\ldots )$ are such that $x_1^-$ is odd,
$x_2^-$ is even, $x_3^-$ is odd and so on. This proves the
condition on the parity for the configurations in
$\Conf^{L}(\Z+\frac{1}{2})$. Moreover, using the definition of
Pfaffian it is not hard to conclude that $\PF\;L\biggl[X_-'\vert
X_-'\biggr]=1$ for the configurations with non-zero probabilities.

Since $\vert X_-''\vert =\vert X_+'\vert +\vert X_+''\vert $ the
matrix $L\biggl[ X_{-}''\sqcup \left(X_{+}'\uplus
X_{+}''\right)\arrowvert  X_{-}''\sqcup \left(X_{+}'\uplus
X_{+}''\right)\biggr]$ has the block structure:
\begin{equation}
\biggl[\begin{array}{cc}
  \mathbb{O}_{d\times d} & Q_{d\times d} \\
  -Q^T_{d\times d} & \mathbb{O}_{d\times d} \\
\end{array}\biggr]
\nonumber
\end{equation}
with
 \begin{equation}
 Q_{d\times d}=\left[\begin{array}{ccccc}
  \dfrac{h(x_1^-)h(x_1^+)}{x_1^--x_1^+} & \dfrac{h(x_1^-)h({^lx}_1^+)}{x_1^--{^lx}_1^+} &
   \ldots & \dfrac{h(x_1^-)h(x_{d/2}^+)}{x_1^--x_{d/2}^+} & \dfrac{h(x_1^-)h({^lx}_{d/2}^+)}{x_1^--{^lx}_{d/2}^+} \\
  \vdots &  & & &  \\
  \dfrac{h(x_d^-)h(x_1^+)}{x_d^--x_1^+} & \dfrac{h(d_1^-)h({^lx}_1^+)}{x_d^--{^lx}_1^+} &
   \ldots & \dfrac{h(x_d^-)h(x_{d/2}^+)}{x_d^--x_{d/2}^+} &  \dfrac{h(x_d^-)h({^lx}_{d/2}^+)}{x_d^--{^lx}_{d/2}^+}\\
\end{array}\right],\nonumber
\end{equation}
where $^lx=x-1$ ($d$ is even). Thus we have
\begin{equation}
\begin{split}
\PF\;L(X\vert X)&=(-)^{\frac{d(d-1)}{2}}
\PF\;\left[\begin{array}{cc}
  \mathbb{O}_{d\times d} & Q_{d\times d} \\
  -Q^T_{d\times d} & \mathbb{O}_{d\times d} \\
\end{array}\right]= \mbox{det}\;Q_{d\times
d}\\
&=(-)^{\frac{d}{2}}(-)^{\frac{d(d-1)}{2}}\;\dfrac{V(\TX_-)V(\TX_+)}{\prod(\TX_+;\TX_-)}\;h(\TX)
\end{split}\nonumber
\end{equation}
where we have used the formula for the Cauchy determinant. Noting
that $(-)^{\frac{d(d-1)}{2}+\frac{d}{2}}=(-)^{\frac{d^2}{2}}=1$
(as
$d$ is even) we obtain the formula stated in the Theorem.\\
\textbf{Case 2.} $X\cap \frak{x}\neq 0$\\
The proof is very similar. We observe that any configuration $X$
has a form
\begin{equation}
X=X_-\sqcup\frak{x}\sqcup X_+ \nonumber
\end{equation}
Then
\begin{equation}
\begin{split}
&\Pf\;L(X\vert X)=\PF\;L\biggl[X_-\sqcup\frak{x}\sqcup X_+\vert
X_-\sqcup\frak{x}\sqcup
X_+\biggr]=\\
&\Pf\;L\biggl[\left(X_-'\uplus X_{-}''\right)\sqcup
(\frak{x}',\frak{x}'')\sqcup \left(X_{+}'\uplus
X_{+}''\right)\vert \left(X_-'\uplus
X_{-}''\right)\sqcup (\frak{x}',\frak{x}'')\sqcup\left(X_{+}'\uplus X_{+}''\right)\biggr]\\
&=(-)^{\frac{d(d-1)}{2}}\cdot \PF\;L\biggl[X_-',\frak{x}'\vert X_-',\frak{x}'\biggr]\\
&\times \PF\;L\biggl[ X_{-}''\sqcup \frak{x}''\sqcup
\left(X_{+}'\uplus X_{+}''\right)\vert  X_{-}''\sqcup
\frak{x}''\sqcup\left(X_{+}'\uplus X_{+}''\right)\biggr]
\end{split}
\nonumber
\end{equation}
Clearly, $\vert X_-''\vert=\vert\frak{x}'\sqcup\left(X_{+}'\uplus
X_{+}''\right)\vert$, otherwise $\Pf\;L(X\vert X)=0$. Thus $\TX_-$
consists of odd number of elements, and
$\vert\TX_-\vert=\vert\TX_+\vert=d$, $d$ is odd, and we repeat the
same computations as in the previous case.

\end{proof}
\begin{cor}\label{CorollaryL-ensembles}
Assume that the nonnegative function $h$ in the definition of the
$L$-matrix (see equations (\ref{PstructureL})-(\ref{Pepsilon})) is
chosen in such a way that
$$
\sum\limits_{X:
X\in\Conf(\Z+\frac{1}{2})_0}\dfrac{V(\TX_-)V(\TX_+)}{\prod(\TX_+;\TX_-)}\;h(\TX)<\infty.
$$
Then the $L$-matrix given by equations
(\ref{PstructureL})-(\ref{Pepsilon}) defines a Pfaffian
$L$-ensemble. Namely, we have
\begin{equation}\label{PReductionToPfaffianProcessesEquation}
\mbox{Prob}_{L}(X)=\dfrac{1}{\Pf\;(J+L)}\;\dfrac{V(\TX_-)V(\TX_+)}{\prod(\TX_+;\TX_-)}\;h(\TX)
\end{equation}
for $X\in \Conf^{L}(\Z+\frac{1}{2})$ and $0$ for all other $X\in
\Conf(\Z+\frac{1}{2})_0$.
\end{cor}
\section{The Plancherel measures with the Jack parameters $\theta=\frac{1}{2}, 2$
 as $L$-ensembles}\label{SectionPlancherelLebsenbles}
\subsection{Expression of the Plancherel measures with the Jack parameters $\theta=\frac{1}{2}, 2$
in terms of the Frobenius-type coordinates}
\begin{prop}\label{PropositionHookFrobenius}
Let $\lambda$ be a Young diagram, and let $(P_1,\ldots,
P_D|Q_1,\ldots, Q_D)$ be the Frobenius  of $\lambda\sqcup\lambda$
(see Section \ref{SectionMainResults}). We have
$$
\frac{1}{H(\lambda,\theta=2)H'(\lambda,\theta=2)}=\frac{\prod\limits_{1\leq
i<j\leq
D}(P_i-P_j)(Q_i-Q_j)}{\prod\limits_{i=1}^D\prod\limits_{j=1}^D(P_i+Q_j+1)\prod\limits_{i=1}^DP_i!Q_i!}.
$$
\end{prop}
\begin{proof}
In a given Young diagram $\lambda$ we consider the diagonal
$j=2i$. There are two possible cases which are distinct from each
other  whether or not the box $(d,2d)$ belongs to the Young
diagram, see Figure 1 and Figure 2.
\begin{figure}
\begin{picture}(100,150)
\multiput(-110,110)(30,0){7}{\framebox(30,30){}}
\multiput(-110,110)(30,0){4}{\framebox(30,30){X}}
\multiput(-110,80)(30,0){5}{\framebox(30,30){}}
\multiput(-110,80)(30,0){4}{\framebox(30,30){X}}
\multiput(-110,50)(30,0){3}{\framebox(30,30){}}
\multiput(-110,20)(30,0){2}{\framebox(30,30){}}
\put(-110,140){\line(2,-1){120}}
\put(-120,90){$d$}
\put(-10,145){$2d$}
\end{picture}
\caption{The box $(d,2d)$ belongs to the Young diagram.}
\end{figure}
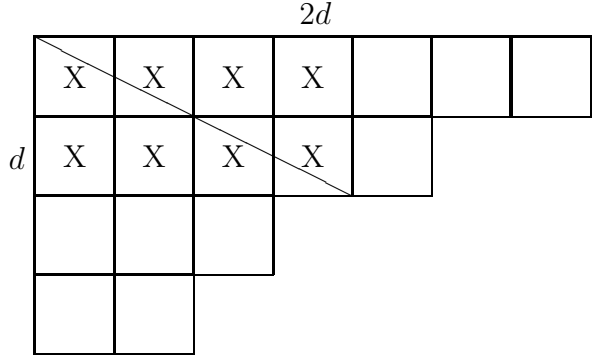
\begin{figure}
\begin{picture}(100,150)
\multiput(-110,110)(30,0){7}{\framebox(30,30){}}
\multiput(-110,110)(30,0){5}{\framebox(30,30){X}}
\multiput(-110,80)(30,0){5}{\framebox(30,30){}}
\multiput(-110,80)(30,0){5}{\framebox(30,30){X}}
\multiput(-110,50)(30,0){5}{\framebox(30,30){}}
\multiput(-110,50)(30,0){5}{\framebox(30,30){X}}
\multiput(-110,20)(30,0){2}{\framebox(30,30){}}
\put(-110,140){\line(2,-1){150}}
\put(-120,60){$d$}
\put(10,145){$2d-1$}
\end{picture}
\caption{The box $(d,2d)$ does not belong to the Young diagram.}
\end{figure}
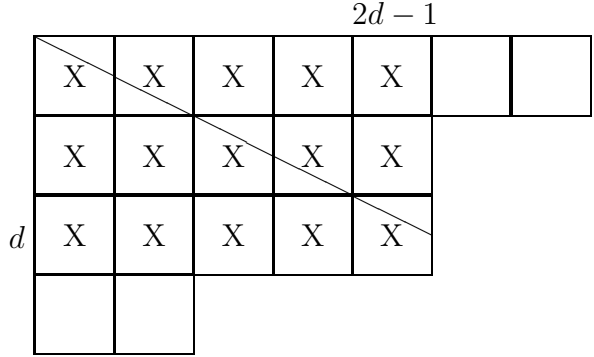
The shape $\lambda$ is divided into three pieces: the rectangular
shape $\lambda^{\Box}$ of size $d\times 2d$ (in the first case
shown on Figure 1), or of size $d\times (2d-1)$ (in the second
case shown on Figure 2); the diagram $\lambda^+$ formed by the
boxes $(ij)$ with $j\geq 2d$ (in the first case), or $j\geq 2d-1$
(in the second case); and the diagram $\lambda^-$ formed by the
boxes $(ij)$ with $i>d$. Thus in both cases the diagram $\lambda$
is composed in the following way
$$
\lambda=\lambda^{\Box}\sqcup\lambda^+\sqcup\lambda^-.
$$
In the subsequent calculations we exploit the following formulae
$$
H(\lambda,\theta)=\prod\limits_{i=1}^{l(\lambda)}\Gamma(\lambda_i-i\theta+l(\lambda)\theta+1)\prod\limits_{1\leq
i<j\leq
l(\lambda)}\frac{\Gamma(\lambda_i-\lambda_j+(j-i)\theta+1-\theta)}{\Gamma(\lambda_i-\lambda_j+(j-i)\theta)},
$$
$$
H'(\lambda,\theta)=\prod\limits_{i=1}^{l(\lambda)}\frac{\Gamma(\lambda_i-i\theta+l(\lambda)\theta+\theta)}{\Gamma(\theta)}\prod\limits_{1\leq
i<j\leq
l(\lambda)}\frac{\Gamma(\lambda_i-\lambda_j+(j-i)\theta)}{\Gamma(\lambda_i-\lambda_j+(j-i)\theta+\theta)}.
$$
These expressions were obtained in Borodin and Olshanski
\cite{BO-ZmeasuresScalingLimits}, and hold true for any
$\theta>0$.

Consider first the case shown on Figure 1. For $1\leq k\leq d$
introduce new coordinates
\begin{equation}\label{8.1.1}
\lambda_k=p_k+2k,\;\lambda_{2k-1}'=\xi_k+k-1,\;
\lambda_{2k}'=\eta_k+k.
\end{equation}
In terms of these coordinates we obtain
\begin{equation}\label{8.1.2}
\frac{1}{H(\lambda^+,\theta=2)H'(\lambda^+,\theta=2)}
=\frac{\prod\limits_{1\leq k<m\leq
d}(p_k-p_m)^2\left((p_k-p_m)^2-1\right)}{\prod\limits_{k=1}^dp_k!(p_k+1)!}.
\end{equation}
Next we use  the formulae stated in Proposition
\ref{PropositionHH} to rewrite the expression
$$
\frac{1}{H(\lambda^-,\theta=2)H'(\lambda^-,\theta=2)}
$$
as follows
\begin{equation}\label{8.1.3}
\begin{split}
&\frac{1}{H(\lambda^-,\theta=2)H'(\lambda^-,\theta=2)}=
\frac{2^{-2|\lambda^-|}\left[\Gamma(\frac{1}{2})\right]^{2d}}{\prod\limits_{k=1}^d\Gamma(\xi_k+\frac{1}{2})\Gamma(\eta_k+1)
\Gamma(\xi_k)\Gamma(\eta_k+\frac{1}{2})}\\
&\times\prod\limits_{1\leq k<m\leq d}(\eta_k-\eta_m)(\xi_k-\xi_m)
\prod\limits_{1\leq 2k-1<2m\leq
d}\left(\xi_k-\eta_m-\frac{1}{2}\right)\prod\limits_{1\leq
2k<2m-1\leq d}\left(\eta_k-\xi_m+\frac{1}{2}\right).
\end{split}
\end{equation}
In addition, we find
\begin{equation}\label{8.1.5}
\begin{split}
H(\lambda^{\Box},\theta=2)=\prod\limits_{i=1}^d\prod\limits_{k=1}^d(p_i+2\xi_k)(p_i+2\eta_k+1),
\end{split}
\end{equation}
and
\begin{equation}\label{8.1.6}
\begin{split}
H'(\lambda^{\Box},\theta=2)=\prod\limits_{i=1}^d\prod\limits_{k=1}^d(p_i+2\xi_k+1)(p_i+2\eta-k+2).
\end{split}
\end{equation}
The new coordinates introduced in equation (\ref{8.1.1}) are
related with the Frobenius coordinates of $\lambda\sqcup\lambda$
as
$$
P=(p_1+1,p_1,\ldots,p_d+1,p_d),\;
Q=(2\xi_1-1,2\eta_1,\ldots,2\xi_d-1,2\eta_d).
$$
We rewrite equations (\ref{8.1.2})-(\ref{8.1.6}) in terms of the
Frobenius coordinates of $\lambda\sqcup\lambda$, and arrive to the
formula stated in the statement of the Proposition. The second
case (shown on Figure 2) can be considered in the same way.
\end{proof}
\begin{prop}\label{PropositionPlancherelFrobenius}
(A) We have
\begin{equation}\label{Pltheta=2Frobenius}
M_{\Plancherel,\eta,\theta=2}(\lambda)=e^{-\eta}(2\eta)^{\frac{1}{2}\sum\limits_{i=1}^D\left(P_i+Q_i+1\right)}
\frac{\prod\limits_{1\leq i<j\leq
D}(P_i-P_j)(Q_i-Q_j)}{\prod\limits_{i=1}^D\prod\limits_{j=1}^D(P_i+Q_j+1)\prod\limits_{i=1}^DP_i!Q_i!}.
\nonumber
\end{equation}
where $\left(P_1,\ldots,P_D\vert Q_1,\ldots,Q_D\right)$ are the
Frobenius coordinates of $\lambda\sqcup\lambda$.\\
(B) We have
\begin{equation}\label{PlthetaHalfFrobenius}
M_{\Plancherel,\eta,\theta=\frac{1}{2}}(\lambda)=e^{-\eta}(2\eta)^{\frac{1}{2}\sum\limits_{i=1}^D\left(P_i'+Q_i'+1\right)}
\frac{\prod\limits_{1\leq i<j\leq
D}(P_i'-P_j')(Q_i'-Q_j')}{\prod\limits_{i=1}^D\prod\limits_{j=1}^D(P_i'+Q_j'+1)\prod\limits_{i=1}^DP_i'!Q_i'!}.
\nonumber
\end{equation}
where $\left(P_1',\ldots,P_D'\vert Q_1',\ldots,Q_D'\right)$ are
the Frobenius coordinates of $\lambda'\sqcup\lambda'$.
\end{prop}
\begin{proof}
These expressions follow from the formula in Proposition
\ref{PropositionHookFrobenius},  and from equations
(\ref{EquationPlancherelInfy})-(\ref{EquationPlancherelInfyMixed}).
\end{proof}
\subsection{Proof of Theorem \ref{MainTheoremTheta2} (B)}
Let $\lambda$ be a Young diagram, and let $(P_1,\ldots,
P_D|Q_1,\ldots, Q_D)$ be the Frobenius coordinates of
$\lambda\sqcup\lambda$.  It is not hard to check that if $X$ is
defined in terms of these Frobenius coordinates by equations
(\ref{X1})-(\ref{X3}), then $X\in\Conf^L(\Z+\frac{1}{2})$ (see
Definition \ref{DefinitionConfigurations}). Conversely, for any
$X\in\Conf^L(\Z+\frac{1}{2})$ there exists a Young diagram
$\lambda$, $\lambda\in\Y$, such that $X=X_-\sqcup X_+$ can be
represented in terms of the Frobenius coordinates of
$\lambda\sqcup\lambda$ as in equations (\ref{X1})-(\ref{X3}). We
conclude that there is a one-to-one correspondence between $\Y$
and $\Conf^{L}(\Z+\frac{1}{2})$, and this correspondence is
defined by equations (\ref{X1})-(\ref{X3}).

Consider the $L$-matrix  defined by equations
(\ref{PstructureL})-(\ref{Pepsilon}), with the weight function $h$
defined by equation (\ref{hPlancherel}). Observe that if the
condition $X\in\Conf^{L}(\Z+\frac{1}{2})$ is not satisfied, then
$\Pf(X|X)=0$. (This follows from the very definition of
$\Conf^L(\Z+\frac{1}{2})$, see Definition
\ref{DefinitionConfigurations}). Therefore it is enough to show
that
\begin{equation}\label{proof0}
M_{\Plancherel,\eta,\theta=2}(\lambda)=\frac{\Pf
L(X|X)}{\Pf(J+L)},
\end{equation}
where $X$ is defined in terms of the Frobenius coordinates of
$\lambda\sqcup\lambda$ as in equations (\ref{X1})-(\ref{X3}). We
use Proposition \ref{PropositionPlancherelFrobenius} (A), and
rewrite $M_{\Plancherel,\eta,\theta=2}(\lambda)$ in terms of the
coordinates $\TX$ as
$$
M_{\Plancherel,\eta,\theta=2}(\lambda)=
e^{-\eta}\frac{V(\TX_-)V(\TX_+)}{\prod(\TX_+,\TX_-)}h(\TX).
$$
(The coordinates $\TX$ are constructed in terms of the coordinates
$X$ as it is explained in Section \ref{SectionSpecialClass}). Then
Proposition \ref{PTheoremReductionToPfaffianProcesses} implies
that
$$
M_{\Plancherel,\eta,\theta=2}(\lambda)= e^{-\eta}\Pf L(X|X).
$$
Since $M_{\Plancherel,\eta,\theta=2}$ is a probability measure on
$\Y$, we have
\begin{equation}\label{proofZvezda}
\begin{split}
1&=\sum\limits_{\lambda\in\Y}M_{\Plancherel,\eta,\theta=2}(\lambda)\\
&=e^{-\eta}\sum\limits_{X:\;
X\in\Conf^{L}(\Z+\frac{1}{2})}\Pf L(X|X)\\
&=e^{-\eta}\sum\limits_{X:\; X\in\Conf(\Z+\frac{1}{2})_0}\Pf
L(X|X).
\end{split}
\end{equation}
(In the last equation we have used the fact that $\Pf L(X|X)=0$
for all $X\in\Conf(\Z+\frac{1}{2})_0$ such that the condition
$X\in\Conf^{L}(\Z+\frac{1}{2})$ is not satisfied). This shows that
$$
\sum\limits_{X:\; X\in\Conf(\Z+\frac{1}{2})_0}\Pf
L(X|X)=e^{\eta}<\infty.
$$
Recall that $\Pf(J+L)$ is defined as the sum $\sum\limits_{X:\;
X\in\Conf(\Z+\frac{1}{2})_0}\Pf L(X|X)$ provided that this sum is
finite. Therefore $\Pf(J+L)=e^{\eta}$, and formula (\ref{proof0})
holds true. \qed
\subsection{Proof of Theorem \ref{MainTheoremThetaHalf} (B)}
Let $\lambda$ be a Young diagram, and let $(P_1',\ldots,
P_D'|Q_1',\ldots, Q_D')$ be the Frobenius coordinates of
$\lambda'\sqcup\lambda'$. Then equations (\ref{X1'})-(\ref{X3'})
define a one-to-one correspondence between $\Y$ and
$\Conf^L(\Z+\frac{1}{2})$. Define the $L$-matrix as in the proof
of Theorem \ref{MainTheoremTheta2} (B) (i.e. by equations
(\ref{PstructureL})-(\ref{Pepsilon}), with the weight function $h$
defined by equation (\ref{hPlancherel})). We need to show that
\begin{equation}\label{proof0'}
M_{\Plancherel,\eta,\theta=\frac{1}{2}}(\lambda)=\frac{\Pf
L(X'|X')}{\Pf(J+L)},
\end{equation}
where $X'$ is defined in terms of the Frobenius coordinates of
$\lambda'\sqcup\lambda'$ as in equations (\ref{X1'})-(\ref{X3'}).
We use  Proposition \ref{PropositionPlancherelFrobenius} (B), and
rewrite $M_{\Plancherel,\eta,\theta=\frac{1}{2}}$ as
$$
M_{\Plancherel,\eta,\theta=\frac{1}{2}}(\lambda)=e^{-\eta}\frac{V(\TX'_-)V(\TX'_+)}{\prod(\TX'_+,\TX'_-)}h(\TX'),
$$
where the coordinates $\TX'$ are related to the coordinates $X'$
in the same way as the coordinates $\TX$ are related to the
coordinates $X$, see Section \ref{SectionSpecialClass}. Formula
(\ref{proof0'}) is then obtained by the same argument as in the
proof of Theorem \ref{MainTheoremTheta2} (B). \qed
\section{The mixed $z$-measures  with the Jack parameters $\theta=\frac{1}{2}, 2$ as $L$-ensembles}
\label{SectionZmeausuresLebsenbles}
\subsection{Expression of the  $z$-measures with the Jack parameters $\theta=\frac{1}{2},
2$ in terms of the Frobenius-type coordinates}
\begin{prop}\label{PropositionMFrobeniuszMeasures}
(A) Let $
\lambda\sqcup\lambda=\left(P_1,\ldots,P_D|Q_1,\ldots,Q_D\right) $
be the Frobenius notation for the Young diagram
$\lambda\sqcup\lambda$, where $D$ is the length of the diagonal in
$\lambda\sqcup\lambda$, and $P_i$, $Q_i$ are the Frobenius
coordinates of $\lambda\sqcup\lambda$. The formula for the
$z$-measure with the Jack parameter $\theta=2$ (equation
(\ref{EquationMzztheta})) can be rewritten as follows
\begin{equation}\label{ZmeasureTheta2Frobenius}
\begin{split}
M_{z,z',\theta=2,\xi}(\lambda)&=(1-\xi)^{\frac{zz'}{2}}\xi^{\frac{1}{2}\sum\limits_{i=1}^D(P_i+Q_i+1)}
\frac{\prod\limits_{1\leq
i<j\leq D}(P_i-P_j)(Q_i-Q_j)}{\prod\limits_{i=1}^D\prod\limits_{j=1}^D(P_i+Q_j+1)}\\
&\times\prod\limits_{i=1}^D\frac{[z+1]_{P_i}[z'+1]_{P_i}[-z]_{Q_i}[-z']_{Q_i}}{P_i!Q_i!}.
\end{split}
\nonumber
\end{equation}
(B) Let $
\lambda'\sqcup\lambda'=\left(P_1',\ldots,P_D'|Q_1',\ldots,Q_D'\right)
$ be the Frobenius notation for the Young diagram
$\lambda'\sqcup\lambda'$, where $D$ is the length of the diagonal
in $\lambda'\sqcup\lambda'$, and $P_i'$, $Q_i'$ are the Frobenius
coordinates of $\lambda'\sqcup\lambda'$. The formula for the
$z$-measure with the Jack parameter $\theta=\frac{1}{2}$  can be
rewritten as follows
\begin{equation}\label{ZmeasureThetaHalfFrobenius}
\begin{split}
M_{z,z',\theta=\frac{1}{2},\xi}(\lambda)&=(1-\xi)^{2zz'}\xi^{\frac{1}{2}\sum\limits_{i=1}^D(P_i'+Q_i'+1)}
\frac{\prod\limits_{1\leq
i<j\leq D}(P_i'-P_j')(Q_i'-Q_j')}{\prod\limits_{i=1}^D\prod\limits_{j=1}^D(P_i'+Q_j'+1)}\\
&\times\prod\limits_{i=1}^D\frac{[-2z+1]_{P_i'}[-2z'+1]_{P_i'}[2z]_{Q_i'}[2z']_{Q_i'}}{P_i'!Q_i'!}.
\end{split}
\nonumber
\end{equation}
\end{prop}
\begin{proof}
We start from the formula
$$
M_{z,z',\theta=2,\xi}(\lambda)=
(1-\xi)^{\frac{zz'}{2}}\xi^{|\lambda|}
\frac{(z)_{\lambda,\theta=2}(z')_{\lambda,\theta=2}}{H(\lambda,\theta=2)H'(\lambda,\theta=2)}.
$$
Given a box $(i,2i-1)$ of a Young diagram $\lambda$, consider the
shape formed by the boxes
$$
(i,2i-1),(i+1,2i-1),\ldots,(\lambda_{2i-1}',2i-1);
$$
$$
(i,2i),(i,2i+1),\ldots,(i,\lambda_i);
$$
and
$$
(i+1,2i),(i+2,2i),\ldots,(\lambda_{2i}',2i),
$$
see Figure 3.
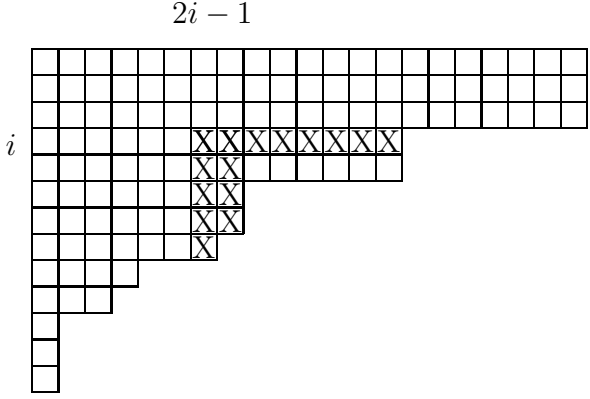
\begin{figure}
\begin{picture}(100,150)
\multiput(-110,110)(10,0){21}{{\dashbox{10}(10,10){}}}
\multiput(-100,100)(10,0){20}{{\dashbox{10}(10,10){}}}
\multiput(-100,90)(10,0){20}{{\dashbox{10}(10,10){}}}
\multiput(-100,80)(10,0){5}{{\dashbox{10}(10,10){}}}
\multiput(-50,80)(10,0){8}{{\dashbox{10}(10,10){X}}}
\multiput(-100,70)(10,0){5}{{\dashbox{10}(10,10){}}}
\multiput(-100,60)(10,0){5}{{\dashbox{10}(10,10){}}}
\multiput(-100,50)(10,0){5}{{\dashbox{10}(10,10){}}}
\multiput(-100,40)(10,0){5}{{\dashbox{10}(10,10){}}}
\multiput(-100,30)(10,0){3}{{\dashbox{10}(10,10){}}}
\multiput(-100,20)(10,0){2}{{\dashbox{10}(10,10){}}}
\multiput(-50,80)(0,-10){5}{{\dashbox{10}(10,10){X}}}
\multiput(-40,80)(0,-10){4}{{\dashbox{10}(10,10){X}}}
\multiput(-30,70)(10,0){6}{{\dashbox{10}(10,10){}}}
\put(-120,80){$i$}
\put(-57,130){$2i-1$}
\multiput(-110,100)(0,-10){12}{{\dashbox{10}(10,10){}}}
\end{picture}
\caption{The distinguished shape associated with the box
$(i,2i-1)$.}
\end{figure}
The contribution of this shape to
$(z)_{\lambda,\theta=2}(z')_{\lambda,\theta=2}$ is
$$
(z)(z-2)\ldots (z-2(\lambda_{2i-1}'-i))(z')(z'-2)\ldots
(z'-2(\lambda_{2i-1}'-i))
$$
$$
\times(z+1)(z+2)\ldots (z+(\lambda_i-2i+1))(z'+1)(z'+2)\ldots
(z'+(\lambda_i-2i+1))
$$
$$
\times(z-1)(z-3)\ldots (z+1-2(\lambda_{2i}'-i))(z'-1)(z'-3)\ldots
(z'+1-2(\lambda_{2i}'-i)).
$$
This can be rewritten in terms of the Frobenius coordinates of
$\lambda\sqcup\lambda$. Observe that the following relations hold
true
$$
P_{2i-1}=\lambda_i-2i+1,\; Q_{2i-1}=2\lambda_{2i-1}'-2i+1,\;
Q_{2i}=2\lambda_{2i}'-2i.
$$
Using these relations we find
$$
(z)_{\lambda;\theta=2}(z')_{\lambda,\theta=2}=\prod\limits_{i=1}^D[z+1]_{P_i}[z'+1]_{P_i}[-z]_{Q_i}[-z']_{Q_i}.
$$
Now we apply Proposition \ref{PropositionHookFrobenius}, and get
the formula for $M_{z,z',\theta=2,\xi}(\lambda)$. The formula for
$M_{z,z',\theta=\frac{1}{2},\xi}(\lambda)$ follows from the
relation
$$
M_{z,z',\theta=\frac{1}{2},\xi}(\lambda)=M_{-2z,-2z',\theta=2,\xi}(\lambda').
$$
This relation is a simple consequence of Proposition
\ref{PropositionMSymmetries}.
\end{proof}
\subsection{Proof of Theorem \ref{MainTheoremTheta2} (A)}
We know (see the proof of Theorem \ref{MainTheoremTheta2} (B))
that equations (\ref{X1})-(\ref{X3}) (expressing each
$X\in\Conf^L(\Z+\frac{1}{2})$ in terms of the Frobenius
coordinates of $\lambda\sqcup\lambda$) define a one-to-one
correspondence between $\Y$ and $\Conf^L(\Z+\frac{1}{2})$. Define
the $L$-matrix by equations (\ref{PstructureL})-(\ref{Pepsilon}),
with the weight function $h$ given by equation
(\ref{hzmeasureTheta=2}). If the condition
$X\in\Conf^L(\Z+\frac{1}{2})$ is not satisfied, then $\Pf
L(X|X)=0$. Therefore it is enough to show that
\begin{equation}\label{proof0z}
M_{z,z',\xi,\theta=2}(\lambda)=\frac{\Pf L(X|X)}{\Pf(J+L)},
\end{equation}
where $X$ is defined in terms of the Frobenius coordinates of
$\lambda\sqcup\lambda$ as in equations (\ref{X1})-(\ref{X3}). We
use Proposition \ref{PropositionMFrobeniuszMeasures}, and rewrite
$M_{z,z',\xi,\theta=2}$ in terms of the coordinates $X$ as
\begin{equation}\label{Zvezdaz}
M_{z,z',\xi,\theta=2}(\lambda)=(1-\xi)^{\frac{zz'}{2}}\frac{V(\TX_-)V(\TX_+)}{\prod(\TX_+,\TX_-)}h(\TX).
\end{equation}
By the same argument as in the proof of Theorem
\ref{MainTheoremTheta2} (B) we obtain equation (\ref{proof0z})
(with $\Pf(J+L)=(1-\xi)^{-\frac{zz'}{2}}$) from formula
(\ref{Zvezdaz}).\qed
\subsection{Proof of Theorem \ref{MainTheoremThetaHalf} (A)}
We use equations (\ref{X1'})-(\ref{X3'}) (expressing each
$X'\in\Conf^L(\Z+\frac{1}{2})$ in terms of the Frobenius
coordinates of $\lambda'\sqcup\lambda'$) to define a one-to-one
correspondence between $\Y$ and $\Conf^{L}(\Z+\frac{1}{2})$.
Observe that formula (\ref{ZmeasureThetaHalfFrobenius}) can be
rewritten in terms of the coordinates $X'$ as
$$
M_{z,z',\xi,\theta=\frac{1}{2}}(\lambda)=(1-\xi)^{2zz'}\Pf
L(X'|X'),
$$
where the $L$-matrix is defined as in the statement of Theorem
\ref{MainTheoremThetaHalf} (A). This follows from Proposition
\ref{PropositionMFrobeniuszMeasures} (B), Proposition
\ref{PTheoremReductionToPfaffianProcesses}, and equations
(\ref{X1'})-(\ref{X3'}). By the same argument as in the proof of
Theorem \ref{MainTheoremThetaHalf} (B) we obtain that
$\Pf(J+L)=(1-\xi)^{-2zz'}$. Therefore,
$$
M_{z,z',\xi,\theta=\frac{1}{2}}(\lambda)=\frac{\Pf
L(X'|X')}{\Pf(J+L)},
$$
i.e. $M_{z,z',\xi,\theta=\frac{1}{2}}$ defines a Pfaffian
$L$-ensemble. \qed

\end{document}